\documentclass[preprint, number, sort&compress]{elsarticle}
\usepackage[T1]{fontenc}
\usepackage{amsmath}%
\usepackage{amsthm}
\usepackage{amsxtra}%
\usepackage{amsfonts}%
\usepackage{amssymb}%
\usepackage{color,hyperref}
\usepackage{graphicx}
\usepackage{geometry}
\hypersetup{colorlinks,breaklinks,
             linkcolor=blue,urlcolor=blue,
             anchorcolor=blue,citecolor=blue}
\usepackage{color}
\usepackage{array}
\usepackage{bbm}
\usepackage{tikz}

\renewcommand{\bar}[1]{{\overline{#1}}}

\newcommand{\E}{{\mathbb E}}

\newcommand{\R}{\mathbb{R}}
\newcommand{\N}{\mathbb{N}}

\renewcommand{\tilde}[1]{\widetilde{#1}}
\renewcommand{\phi}{\varphi}

\def\XXint#1#2#3{{\setbox0=\hbox{$#1{#2#3}{\int}$ }
\vcenter{\hbox{$#2#3$ }}\kern-.6\wd0}}

\newtheorem{theorem}{Theorem}
\newtheorem{lemma}[theorem]{Lemma}
\newtheorem{corollary}[theorem]{Corollary}

\theoremstyle{definition}
\newtheorem{remark}[theorem]{Remark}

\title{Games on deBruijn Graphs and Cycle Means\tnoteref{t1}}
\tnotetext[t1]{The author gratefully acknowledges support by NSF grant DMS 2407839.}
\author{Nadejda Drenska\corref{cor1}}%
\address{Department of Mathematics, Louisiana State University, USA}%

\cortext[cor1]{Email address: ndrenska@lsu.edu.}

\begin{document} 
\begin{abstract}
deBruijn graphs are widely used in genomics and computer science. In this paper we present a novel approach to finding weights on edges of doubly weighted deBruijn graphs. Given any fixed set of weights on vertices, we use a repeated two-person zero-sum game to find weights on edges so that every cycle on the deBruijn graph has the same average weight, providing explicit formulas. This approach uses minimax optimal strategies of the players. Once the weights on the edges are determined, we observe that they correspond to solving a set of linear equations with as many equations as there are cycles. This is very surprising, because there are many more cycles than unknowns. Moreover we analyze other, related games on graphs.
\end{abstract}

\maketitle
\section{Introduction}

deBruijn graphs  \cite{deBruijn, Good, Fredericksen} are regular, directed graphs, belonging to the family of graphs on alphabets. They are widely used in genomics in deNovo assemblies \cite{Pevzner1, Pevzner2, Zerbino, Chikhi, Iqbal} and bioinformatics \cite{Minkin}. They are also used in computer science for distributed hash tables (e.g. Koorde \cite{Kaashoek}), for routing \cite{Cheng}, and for shape analysis \cite{Ulusoy}. Other applications include distributed control \cite{Delvenne} and game theory \cite{Renault}.
Therefore, understanding the structure of deBruijn graphs  is of considerable interest.

Cycles with the lowest and the highest arithmetic means among all cycles on a graph are known as minimal and maximal mean cycles, or optimal cycles. As outlined for example in  \cite{Dasdan, Cuninghame, Maji}, optimal cycles are important for analysis of digital systems, for instance rate analysis of embedded systems. 
In \cite{Maji} the author considers minimum balancing of a directed graph (situations where optimal cycles are analyzed and weights on edges are adjusted), in order to optimize for clock networks performance. According to \cite{Dasdan2}, some applications are in "cycle time and slack optimization for circuits, retiming, timing separation analysis, and rate analysis." \cite{Albrecht} states that  cycles are used in logic chips.

Doubly weighted graphs (graphs with weights on both vertices and edges) appear in various papers, for example \cite{Golitschek, Golitschek2, Shi}.  Key works on the analysis of optimal cycles on doubly weighted graphs are in \cite{Golitschek, Golitschek2}. In \cite{Golitschek}, the author relates optimal cycles of doubly weighted graphs to the approximation of a function of two variables by two functions of one variable, whereas in \cite{Golitschek2} the author makes connections with solving linear systems of inequalities.  

In this paper we analyze the structure of doubly weighted deBruijn graphs and optimal cycles on them. In particular, we use adversarial repeated two person games on a graph. Given all vertex weights, we prove the existence of an assignment of weights on edges, so that every cycle on the deBruijn graph has the same average weight (see Theorem \ref{main} for further description). This is related to solving a linear system of equations on a deBruijn graph. 

Inspiration for the current paper came from a problem considered in \cite{drenska2020pde}, and later on resolved in \cite{calder2021asymp}. These two papers include a game with two adversarial players, which is partially played on a deBruijn graph on two symbols; the weights on vertices are functions of time, space, and vertices on a graph. In the analysis, one player seeks minimal weight accumulation in the long run, which is related to solving a \textit{time-independent problem} with time-independent weights, namely finding weights on edges, so to minimize the heaviest cycle. In this paper we solve a related \textit{time-dependent problem}, namely finding weights on edges so to minimize the heaviest path, which does not depend on a spatial variable, on a deBruijn graph on $n\geq 2$ symbols. In \cite{drenska2020pde} and \cite{calder2021asymp}, the two players have fixed roles (one picks weights, the other - edges), whereas in the general game described in Chapter \ref{related}, the two players have equal roles: on any turn each of the two players can pick weights. 

This paper is organized as follows. First we introduce the problem, then we discuss a game associated to the problem. The solution to this game produces a surprising result, associated to overdetermined systems of linear equations on graphs. Next, we consider two related games, the second game being a generalization of the other one and of the original one. All these games have the same solution. Moreover, this solution solves a discrete Poisson equation.

\section{Notation and Background}\label{notation}
In this section we introduce the deBruijn graph and the game associated to it. 
\subsection{The Graph}\label{graph}
Let $d$ and $n$ be natural numbers. We construct $G$ --- a deBruijn graph on $n$ symbols, $\{0, 1, \dots, n-1\},$ and length of words $d$ --- as follows: 
\begin{itemize}
\item Each vertex is identified through a $n$-ary representation on $d$ symbols. So, there are $N=n^d$ vertices. We can think of each vertex as being identified through a number $0, 1, ..., n^d-1$ (see Figure \ref{f1}). We denote the set of vertices to be$$\mathcal{M}:=\{0,1,\dots, n^d-1 \} =\{0, 1, \dots, n-1\}^d.$$
In the example on Figure \ref{f1} we have that $n=2$ and $d=3.$
\item We denote generic vertices on the graph with $m, \bar{m}\in \mathcal{M}.$
\item We denote the starting vertex with $m_{b}$ and the end vertex with $m_e.$
\item On every vertex $m$, there is a weight function $c: \mathcal{M} \to \R.$ 
\item Every vertex has precisely $n$ outgoing and $n$ incoming edges. For any vertex $m \in \mathcal{M}$ we form an edge from $m$ to the $n$ vertices obtained by removing the first digit of $m$ and concatenating one of the digits $0, 1, \dots, n-1.$ See Figure \ref{f1} for an example.
\item We use the following notation for this: if a vertex is obtained from vertex $m$ by appending the digit $\ell$ in its n-ary representation, the new vertex is denoted by $m|\ell.$ 
\item The edge between two vertices, where  $m_1$ is the source and  $m_2$ is the sink, is denoted by $(m_1, m_2).$ In particular, the edge between $m$ and $m|\ell$ is denoted by $(m, m|\ell).$
\item Each edge on the graph will have various weights associated to it, as described in Section \ref{game}.
\item Consider any weights assignment of vertices and edges on $G.$ We define the weight of a path to be the sum of the weights of the vertices and plus the sum of the weights of the edges along the path. We include both the start and the end vertices' weights. 
\item We define the weight of a cycle to be the sum of the weights of all vertices of the cycle plus the sum of the weights of all edges of the cycle. Every vertex and every edge of the cycle are counted exactly once.
\item For any cycle, the average weight of a cycle is the weight of the cycle divided by the number of vertices of the cycle.
\end{itemize}

\begin{figure}\label{f1}
\begin{center}
\begin{tikzpicture}[scale=1.5]
\node at (-1,0) {010};
\node at (1,0) {101};
\node at (-3,0) {000};
\node at (3,0) {111};
\node at (-2,-1) {100};
\node at (-2,1) {001};
\node at (2,-1) {110};
\node at (2,1) {011};
\draw[thick,->] (-3.2,0.15) arc (25:335:0.3);
\draw[thick,->] (3.2,-0.15) arc (-180+25:180-25:0.3);
\draw[thick,->] (-2.8,0.1)--(-2.2,0.85);
\draw[thick,<-] (-2,0.85)--(-2,-0.85);
\draw[thick,->] (-1.8,0.85)--(-1.2,0.1);
\draw[thick,<-] (-1.8,-0.85)--(-1.2,-0.1);
\draw[thick,<-] (-2.8,-0.1)--(-2.2,-0.85);
\draw[thick,<-] (2.8,0.1)--(2.2,0.85);
\draw[thick,->] (2.8,-0.1)--(2.2,-0.85);
\draw[thick,->] (2,0.85)--(2,-0.85);
\draw[thick,<-] (1.8,0.85)--(1.2,0.1);
\draw[thick,->] (1.8,-0.85)--(1.2,-0.1);
\draw[thick,->] (-1.8,1)--(1.8,1);
\draw[thick,<-] (-1.8,-1)--(1.8,-1);
\draw[thick,->] (0.85,0.15) arc (90-25:90+25:2);
\draw[thick,->] (-0.85,-0.15) arc (270-25:270+25:2);
\end{tikzpicture}
\end{center}
\caption{The de Bruijn Graph, $ n=2, d=3$}
\label{fig:graph}
\end{figure}
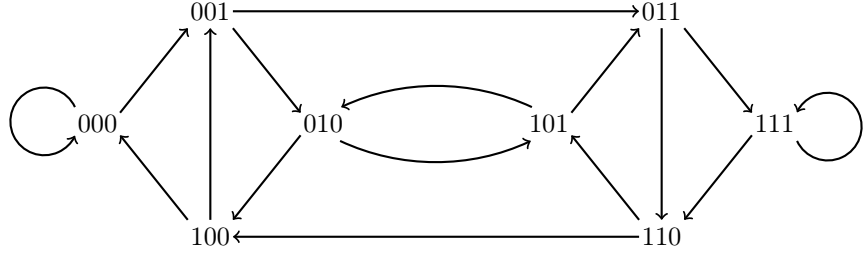

\subsection{The Game} \label{game}
In order to answer the main question in this paper (see Section \ref{main_result}), we analyze a repeated two-person zero-sum game played on the graph. In sections \ref{related} and \ref{general_graphs} we also look at other, related games. The rules of the original game are as follows.
\begin{itemize}
\item The game is played on the deBruijn graph $G$ described in Section \ref{graph}.
\item There are two players, Paul and Carol. 
\item The game is played starting at turn $0$ until the final turn $T,$ with $T\in \N,$ fixed a priori.
\item The current turn (or time) is denoted by $t \in \{ 0, 1, \dots, T \}.$


\item There is a token of the game; the game starts at vertex $m_{b}\in \mathcal{M}$  and time zero. 
\item When the turn of the game reaches time $T,$ it stops.
\item At every turn, which we parameterize by the turn number $t$ and state $m$ on the graph, Paul chooses the weights along the $n$ outgoing edges between $m$ and $m|0, m|1, \dots, m|(n-1).$  These weights are real numbers.
\item Let the weight per edge $(m,m|\ell)$ at turn $t$ be denoted by $f(t, (m,m|\ell)).$ 
\item Paul's choice at turn $t$ is subject to the restriction that the sum of all edges' weights at time $t$ originating from the current vertex $m$ has to equal zero:
\begin{equation}\label{f_sum}
\sum_{\ell=0}^{n-1}f(t, (m,m|\ell))=0.
\end{equation}
\item Having seen Paul's choice of weights, Carol picks, deterministically or probabilistically, a vertex among $m|0, m|1, \dots, m|(n-1)$ and the token of the game moves to the chosen vertex.
\item This way, the token of the game moves continuously along the path chosen by Carol. 
\item As the game progresses, Carol is allowed to choose a vertex that has already been visited. Then, Paul is allowed to choose a different weight assignment from the one(s) used in the past.
\item There is no restriction on whether weights can be repeated.
\item Paul and Carol play against each other adversarially. The cost of the total path traversed by the game token is analyzed. Paul's goal is to assign weights such that the cost of the path is as small as possible. Carol's goal is to choose vertices such that the cost of the path is as large as possible.
\item When Paul and Carol are behaving optimally, the cost of the total path traversed by the game token is determined by the length of the game $T$ and the beginning vertex $m_b$. We refer to the cost of an optimally-played, length-$T$ game, beginning at vertex $m_b$, as $U(T, m_b)$.
\item When the game is played optimally, at turn $t$ and vertex $m$, players Paul and Carol respectively minimize and maximize the cost that will be accrued over the remaining $T-t$ turns of the game, which we call $v(t, m)$. This to-be-accrued cost is equal to the cost of a game of length $T-t$ that starts at vertex $m$. Hence $U(T-t, m) = v(t, m)$.
\item We call $v$ the value function of the game.
\item If there are no turns left to play, the cost of the remaining path has to be the cost of the final vertex. In other words,
$$v(T,m)=c(m)$$
for all $m\in \mathcal{M}.$
\end{itemize}



\section{Main Result}\label{main_result}
The main question we would like to address in this paper is as follows. Given a deBruijn graph and fixed weights on its vertices, is it possible to find weights on edges so that:

1) for every vertex, the set of edges coming out of this vertex has total weight that is equal to zero.

2) the average weight of every cycle on this graph is constant. 

Answer: yes, such an assignment exists!

As a direct corollary, we see that there exist graphs of size $n^d,$ for $n\in \N, d\in \N$ such that no matter what real vertex weights are chosen, there exists an edge assignment such that the minimum mean cycle ratio and the maximum mean cycle ratio coincide!

\section{The Value Function}\label{valuefn}
Let $v(t,m_{b})$ be the value function of the game if both Paul and Carol were playing optimally against each other, starting at vertex $m_{b}$ and turn $t$ (therefore with $T-t$ steps left to go) and following the rules outlined in Subsection \ref{game}. Paul minimizes the sum of weights of any walk on the graph that starts at $m_b$ and lasts $T-t$ turns. Carol chooses the $T-t$ vertices that constitute the heaviest walk that Paul's choices of weights permit. This makes this game an adversarial repeated two-person game.
As a final time condition at the end of the game (meaning at time $T$ and an end vertex $m_e$), we set 
$$ v(T,m_e)=c(m_e).$$
The value function for any generic time $t$ with $0\leq t \leq T-1$ and generic vertex $m$ is $T-t$ successive minimax-es of the sum of the vertex and edge weights for the last $T-t$ turns of the game:
{\footnotesize
\begin{equation}\label{v_def}
v(t,m)=\min_{f(t,(m,m|\ell_t)) } \max_{\ell_t} \dots \min_{f(T-1,(m_{T-2}, m_{T-1})) 
}\max_{\ell_{T-1}}\Big\{\sum_{s=t}^T c(m_s)+\sum_{s=t}^{T-1} f(s,(m_{s-1},m_s)) \Big\},
\end{equation}
}
where $m_0=m$ and $m_s=m|\ell_t|\ell_{t+1}|\dots|\ell_s.$ The minima are subject to constraints, namely \eqref{f_sum}, that is: the first minimum is subject to $\sum f(t,(m,m|\ell_t))=0,$ whereas the last minimum is subject to $\sum f(T-1,(m_{T-2}, m_{T-1}))=0.$


Since we only consider optimal strategies, the value $v$ of the game is uniquely determined by the optimal choices of Paul and Carol. One can see this through a recursive interpretation of the game. Suppose the game is at state $m$ and turn $t$ with $t<T.$ Suppose Paul knew the values of the game for $t+1$ and for all $\tilde{m}$ that are adjacent to $m.$ Then Paul considers the optimal costs at turn $t+1,$ namely $v(t+1,\tilde{m}),$ and builds the value function $v(t,m)$ using these optimal functions. 
Thus we can think of $v(t,m)$ being constructed through an optimal action at $(t+1, \tilde{m}).$
In other words, we have a recursive definition of the value function $v(t,m)$ at time $t$ through a minmax of the value functions $v(t+1, m|\ell)$ at time $t+1$ for all possible successors $m|\ell$ of $m$. Therefore the minmax is global. This becomes more transparent in formula \eqref{DPP_formula}, the Dynamic Programming Principle.

We recall that for every $t$ and every generic vertex $\bar{m}\in\mathcal{M}$ we have formula \eqref{f_sum}: 
$$\sum_{\ell=0}^{n-1}f(t, (\bar{m},\bar{m}|\ell))=0.$$ 
So, the choice that Carol makes at turn $t$ is to pick a vertex $m|\ell$ as to maximize the cost of the current vertex $c(m)$ plus the weight per given edge $(m, m|\ell),$ plus the remaining weight of the path. Therefore, Carol's choice is between options 
$$c(m)+f(t, (m, m|\ell)) + v(t+1, m|\ell), ~~~\text{for}~~~\ell=0, \dots, n-1.$$


In what follows we frequently use the Dynamic Programming Principle (see  \cite{book_Bertsekas} or \cite{book_Cormen}).
Because of the iterative nature of problem \eqref{v_def}, and because the decisions of Paul and Carol alternate, we obtain the following Dynamic Programming Principle formulation
\begin{equation}\label{DPP_formula}
v(t,m)=\min_{f: ~~~~\sum_{\ell=0}^{n-1}f(t, (m,m|\ell))=0} \max_{\ell} \big\{c(m)+ f(t, (m,m|\ell))+ v(t+1, m|\ell) \big\} ~~~\text{if}~~~ t<T,
\end{equation}
with the final time condition
$$v(T,m_e)=c(m_e).$$


Note that, in particular, 
\begin{equation}
v(0,m_b)=\min_{f: ~~~~\sum_{\ell=0}^{n-1}f(t, (m_b,m_b|\ell))=0} \max_{\ell} \big\{c(m_b)+ f(0, (m_b,m_b|\ell))+ v(1, m_b|\ell) \big\} ~~~\text{if}~~~ t<T.
\end{equation}
In \eqref{DPP_formula}, Paul chooses $f(t, (m, m|\cdot))$ on every step $t$, whereas Carol chooses the next vertex, probabilistically or deterministically, along the path by picking $\ell$ to arrive at vertex $m|\ell.$

The final value, for any vertex $m$, is $v(T,m)=c(m).$ By going backwards in time, we define, using the Dynamic Programming Principle, $v(T-1, \bar{m}),\dots, v(0,m_b)$  for any $\bar{m}, m_b.$ The value of the game is defined recursively, because the optimal behaviour of Paul and Carol relates what happens at time $t$ to what would happen at a future turn $t+1$ provided optimal choices. 

\begin{lemma}\label{value}
For every $t\in \mathbb{N}_0$ with $t<T$, the value function fulfills the following:
$$v(t,m)=c(m) + \frac{1}{n}\sum_{\ell=0}^{n-1}v(t+1, m|\ell).$$
\end{lemma}
\begin{proof}
Observe that, for fixed $t,m$ one optimal choice for Carol is to pick the largest value among 
$c(m)+ f(t, (m,m|\ell))+ v(t+1, m|\ell)$ and move the token of the game to $m|\ell$. 

Knowing this strategy, Paul has to minimize the largest such value (subject to \eqref{f_sum}). The minimum occurs when these values are equal, because these are linear expressions.
Therefore, Paul sets 
\begin{equation}\label{v_sum}
c(m)+ f(t, (m,m|\ell))+ v(t+1, m|\ell)=c(m)+ f(t, (m,m|\hat{\ell}))+ v(t+1, m|\hat{\ell})
\end{equation}
for all $\ell, \hat{\ell} \in \{0, 1,\dots, n-1 \}.$
Together with equation \eqref{f_sum}, these are ${ n\choose 2} +1$ linear equations in the $n$ variables $f(t, (m,m|0)),  \dots, f(t,(m,m|n-1)),$ with the unique solution being 
\begin{equation}\label{star}
f(t,(m,m|\hat{\ell})) = \frac{1}{n}\sum_{\ell=0}^{n-1}v(t+1, m|\ell) - v(t+1, m|\hat{\ell}). 
\end{equation}
Upon plugging in \eqref{star} into \eqref{DPP_formula}, we obtain the result.
\end{proof}
\begin{lemma} \label{probab}
Carol's choice might as well be probabilistic, because, if so, the value function won't change.
\end{lemma}
\begin{proof}
Suppose at time $t$ and vertex $m,$ Carol chooses vertex $m|\ell$ with probability $\alpha_\ell(t).$ Then, the expected weight of the path chosen by her strategy would be
\begin{equation}\label{v_alpha}
  \sum_{\ell=0}^{n-1} \alpha_\ell(t) \Big(c(m)+ f(t, (m,m|\ell))+ v(t+1, m|\ell)\Big).
\end{equation}
Here $\alpha_\ell(t)$ are nonnegative numbers such that $\sum_{\ell=0}^{n-1}\alpha_\ell(t)=1.$

Upon analyzing at the Dynamic Programming Principle and using formulas  \eqref{v_sum} and \eqref{star}, we see that
\begin{eqnarray}\label{v_alpha}
 && \sum_{\ell=0}^{n-1} \alpha_\ell(t) \Big(c(m)+ f(t, (m,m|\ell))+ v(t+1, m|\ell)\Big) = \\
  &&=c(m)+  \sum_{\ell=0}^{n-1} \alpha_\ell(t) \Big( f(t, (m,m|\ell))+ v(t+1, m|\ell)\Big)=
  v(t,m),
\end{eqnarray}
so essentially Carol's choice can be probabilistic, too, and this won't change the value function $v(t,m)$ or the solution $v(0,m_b)$ to the problem.
\end{proof}

\begin{theorem}
The explicit formula of the value function is:
$$v(t,m)=c(m)+\frac{1}{n}\sum_{\ell_1=0}^{n-1}c(m|\ell_1)+ \dots +\frac{1}{n^{T-t}}\sum c(m|\ell_1|\ell_2,...|\ell_{T-t}). $$
The second sum is over all $\ell_i\in\{0, 1, \dots, n-1 \}$ for $i=1, 2, \dots, T-t.$
Observe that when $T-t>d,$ we obtain
{
\begin{equation} \label{star_new}
v(t,m)=c(m)+\frac{1}{n}\sum_{\ell_1=0}^{n-1}c(m|\ell_1)+ \dots +\frac{1}{n^{d-1}}\sum c(m|\ell_1|\ell_2|\dots|\ell_d) + \frac{(T-t-d+1)}{n^d}\sum_{\bar{m}\in\mathcal{M}} c(\bar{m}).
\end{equation}
The second sum is over all $\ell_i\in\{0, 1, \dots, n-1 \}$ for $i=1, 2, \dots, d.$
}
\end{theorem}
\begin{proof}
The proof is by induction. 
%
\end{proof}

\begin{lemma}\label{lemma2}
The optimal weights that Paul picks at any step earlier than $T-d$ are independent of time.
\end{lemma}
\begin{proof}
Consider any vertex $m$ and its successor $m|\ell$. Then, the optimal weight picked by Paul is, using equation \eqref{star}, 
\begin{align*}
f(t, (m,m|\ell)) &= \frac{1}{n}\sum_{\hat{\ell}=0}^{n-1}v(t+1, m|\hat{\ell}) - v(t+1, m|\ell) = \frac{1}{n}\sum_{\hat{\ell}=0}^{n-1}v(t, m|\hat{\ell}) - v(t, m|\ell)= \dots \\
 &=\frac{1}{n}\sum_{\hat{\ell}=0}^{n-1}v(s+1, m|\hat{\ell}) - v(s+1, m|\ell) = f(s,(m,m|\ell)),
\end{align*}
where $s<t$ with $s\in\N_0.$
We observe that the above equalities hold for $s, t\leq T-d,$ because they consider the weighted average of what vertices can be reached by the deBruijn tree rooted at $m.$
\end{proof}
Since $f(t,(m,m|\ell))$ does not depend on t for $t<T-d,$ we denote $f(m, m|\ell)=f(t, (m,m|\ell))$ to be this common value.

\begin{corollary}\label{corollary2} 
Consider a cycle, whose traversal ends earlier than time $T-d.$ Then, the average of the sum of weights it takes this traversal is independent of time.
\end{corollary}
\begin{proof}
The statement of the corollary follows from Lemma \ref{lemma2}, using that the weight on each vertex is constant, and the weights on edges are constant for edges that start before time $T-d.$ Since the cycle ends at time before $T-d,$ all edges encountered during a walk that traverses the cycle occur before time $T-d.$ 
\end{proof}

\begin{theorem}\label{main}
For cycles that end before time $T-d,$ every cycle in the graph has the same average weight. This weight equals the arithmetic mean of the weight on every vertex, namely
$$ \frac{1}{n^d}\sum_{\bar{m}\in\mathcal{M}}c(\bar{m}).$$
\end{theorem}
\begin{proof}
We consider a cycle starting at vertex $m$ of length $k.$ By Corollary \ref{corollary2}, the weight to traverse such a cycle is independent of time, provided the edges are traversed early (before time $T-d$). 
We observe that from \eqref{star} and for $s\geq d$ the weight of the cycle is precisely 
$$v(T-s-k, m)-v(T-s, m)=k \frac{1}{n^d}\sum_{\bar{m}\in \mathcal{M}}c(\bar{m}).$$
Therefore, the average weight of the cycle is $ \frac{1}{n^d}\sum_{\bar{m}\in \mathcal{M}} c(\bar{m}),$ which concludes the proof. 
\end{proof}
This resolves the main question of the paper. We see that there exists an assignment of weights on edges, coming from formulas \eqref{star} and \eqref{star_new}, that results in every cycle on the deBruijn graph having the same average weight $ \frac{1}{n^d}\sum_{\bar{m}\in\mathcal{M}}c(\bar{m}).$

Finally, we also have that the value function $v$ solves a discrete Poisson equation on the graph. Note that a similar result is obtained for $n=2$ and for a time-independent problem in \cite{calder2021asymp}. 
Let $\Delta$ be the discrete Laplacian on the deBruijn graph, meaning for any function $h$ defined on the graph,
$$\Delta h(t,m) := h(m) - \sum_{\ell=0}^{n-1}h(t,m|\ell).$$ 
We state the following result.
\begin{theorem}
For $t<T-d$ the value function $v,$ whose formula is \eqref{star},  solves the discrete Poisson equation 
\begin{equation}\label{delta}
\Delta v(t,m) = c(m) - \frac{1}{n^d}\sum_{\bar{m}\in \mathcal{M}}c(\bar{m}).
\end{equation}
\end{theorem}
\begin{proof}
The proof follows from plugging in \eqref{star} in \eqref{delta}.
\end{proof}

\section{Related Games}\label{related}
In this section we present two games, which are related to the main game described in Sections \ref{notation} and \ref{valuefn}. First, in a remark we consider the simpler game. The rest of the subsection is devoted to the second game, which is a generalization of the original game, and of the one in the remark below.
\begin{remark}
Absolutely analogous results to the ones presented in \ref{game} are obtained when playing a game with the objective of Paul to maximize, and Carol to minimize, but otherwise with the same rules. Then the value function can be written as 
$$\tilde{v}(t,m)=\max_{f(t,(m,m|\ell))} \min_{\ell} \big\{c(m)+ f(t, (m,m|\ell))+ \tilde{v}(t+1, m|\ell) \big\} ~~~\text{if}~~~ t<T$$
$$\tilde{v}(T,m_e)=c(m_e),$$
with the maximization subject to the constraint \eqref{f_sum}.

We observe that, using the Dynamic Programming Principle, Carol needs to pick the weights on edges $f$ so that every potential path, if played optimally, has equal weight. This necessitates that formula \eqref{v_sum} holds so Lemma \ref{value} also holds. From there on, the analysis is analogous.  Therefore, for any 
$t\in \N_0, t\leq T$ and $m\in\mathcal{M},$ we have that 
\begin{equation}\label{tilde_equals}
\tilde{v}(t,m)=v(t,m).
\end{equation}
\end{remark}

Here we present a game which is more challenging to analyze. It is one where some of the time Carol, the rest of the time Paul, picks edge weights. 

Consider the game described in Subsection \ref{game} but with the following modification. The rules of the walk are such that on some predetermined turns, known to both Paul and Carol and represented by the set $S\subset \{t, t+1, \dots,T-1 \},$ Paul picks weights and Carol picks an edge then, and the rest of the turns Carol picks weights, and Paul picks an edge then. 
Let $u_S(t, m)$ be the value function for this game. We prove the following theorem:

\begin{theorem}\label{u_S}
Consider  $t\in \N_0, t\leq T$ and $m\in \mathcal{M}.$ Let $S$ be a set of integers, each integer between $t$ and $T.$ Then,
\begin{equation}
u_S(t,m)=v(t,m).
\end{equation}
\end{theorem}


Observe that the rules now have `double' optimization in the sense that in some situations Paul minimizes over picking an edge and assigning a weight distribution to next moves together, in one go (same applies to Carol maximizing). Thus, picking an edge and assigning weights now go together, as one optimization; what used to be a Dynamic Programming Principle statement for $v$ cannot be simply written as one, as it is not necessarily true that $\min_{f(t,m|\ell),\ell}=\min_{f(t,m|\ell)} \min_\ell,$ or $\max_{f(t,m|\ell),\ell}=\max_{f(t,m|\ell)} \max_\ell.$ We will nevertheless prove that for this game's optimal strategy indeed we can break the $\min$ (and the $\max$) into two pieces. 

Let us denote $m_0=m$ and $m_s=m|\ell_t|\ell_{t+1}|\dots|\ell_s.$
\begin{lemma}\label{u_S_eq}
Under the assumptions of Theorem \ref{u_S}, when $t\in S$ we have
{\footnotesize
\begin{align}\label{u_min}
u_S(t,m)&=\min_{f_t} \max_{\ell_t, f_{t+1}}\min\dots \Big\{\sum_{s=t}^T c(m_s)+\sum_{s=t}^{T-1} f(s,(m_{s-1}|m_s)) \Big\} = \\ \nonumber
&\min_{f_t} \max_{\ell_t}\max_{f_{t+1}}\min\dots \Big\{\sum_{s=t}^T c(m_s)+\sum_{s=t}^{T-1} f(s,(m_{s-1}| m_s)) \Big\} 
\end{align} 
}
Similarly, when $t\notin S$
{\footnotesize
\begin{align}\label{u_max}
u_S(t,m)&=\max_{f_t} \min_{\ell_t, f_{t+1}}\max\dots \Big\{\sum_{s=t}^T c(m_s)+\sum_{s=t}^{T-1} f(s,(m_{s-1}|m_s)) \Big\} = \\ \nonumber
&\min_{f_t} \max_{\ell_t}\min_{f_{t+1}}\max\dots \Big\{\sum_{s=t}^T c(m_s)+\sum_{s=t}^{T-1} f(s,(m_{s-1}| m_s)) \Big\} 
\end{align} 
}
\end{lemma}
\begin{proof}
First, consider statement \eqref{u_min}. Observe that the outermost operation is $\min,$ so Paul makes a move first; regardless of what the remaining optimization's values are, Paul knows that Carol will take advantage (pick the heaviest path), should all possible paths' values be different. Thus, Paul's optimal strategy is to make the heaviest path as light as possible, which occurs when the function $f_t$ makes all path's weights equal. As a result, Carol's choice of $\ell_t$ is irrelevant. Therefore, we can break the $ \max_{\ell_t, f_{t+1}}$ into two $\max$-es. The proof of formula \eqref{u_max} is analogous. 
\end{proof}

\begin{lemma}
The Dynamic Programming Principle applies to $u_S.$
\end{lemma}
\begin{proof}
The proof is by induction over $t= T, T-1, T-2, \dots$ The base cases are $T, T-1$ and they follow  trivially. The inductive step follows from Lemma \ref{u_S_eq} and the fact that once the optimization is over an alternating sequence of $f_s$ and $\ell_s,$ we can write $u(t,m)$ through a $\min\max$ (or $\max\min$) and its value $u(t+1,m|\ell).$ Therefore, we obtain the Dynamic Programming Principle. 
\end{proof}

\begin{proof} \textit{of Theorem \ref{u_S}.}
The proof is by induction on $T, T-1,\dots, t, \dots$
For $T, T-1$ the statement is true. Then, we assume that 
\begin{equation}
u_S(t,\tilde{m})=v(t,\tilde{m})
\end{equation}
holds for all $\tilde{m}\in \{0,1,\dots, n-1\}^d$ and $t+1, t+2,\dots, T.$
Let $\tilde{S}=S\setminus t.$

Case 1: Suppose $t\in S.$ 
Then, by the Dynamic Programming Principle and by the inductive step, we obtain
\begin{align}
u_S(t,m) &= \min_{f}\max_{\ell}\{c(m)+f(t, (m,m|\ell))+ u_{\tilde{S}}(t+1,m|\ell) \} = \nonumber \\
 &= \min_{f}\max_{\ell}\{c(m)+f(t, (m,m|\ell))+ v(t+1,m|\ell) \} = \nonumber \\
 &\equiv v(t,m). \nonumber
\end{align}
Case 2: Suppose $t\notin S.$ Then, using the Dynamic Programming Principle and \eqref{tilde_equals}, it follows that
\begin{align}
u_S(t,m) &= \max_{f}\min_{\ell}\{c(m)+f(t, (m, m|\ell))+ u_{\tilde{S}}(t+1,m|\ell) \} = \nonumber \\
  &= \max_{f}\min_{\ell}\{c(m)+f(t,(m,m|\ell))+ \tilde{v}(t+1,m|\ell) \} = \nonumber \\
 &\equiv \tilde{v}(t,m) = v(t,m).\nonumber
\end{align}
The last equality is \eqref{tilde_equals}.
\end{proof}

\section{General Graphs}\label{general_graphs}
A lot of the ideas presented in this paper are not restricted to deBruijn graphs, even though using them makes notation and ideas clearer.
In this section we present a result for more general graphs --- a generalization of the problem to any directed graph without sinks.
Let us denote $G=\{V, E\}.$ Let $V=\{v_0,\dots v_{N-1} \},$ where $N=|V|.$
Let $u(t, v_m)$ be the value function.

Instead of thinking deterministically, we will use probability. We observe that we can introduce a random process that governs the movement of the token of the game. We define this random process as follows: at any vertex, a successor of this vertex is picked uniformly at random. Because of Lemmas \ref{value} and \ref{probab}, we know that the movement of the token is such that, after Paul picks all the weights, the value function  along any path is equal. Therefore, the random process is: at any vertex, the token of the game is moved by Carol with equal probability to a neighbouring outgoing vertex.  

Let us denote with $C_{k,m}$ a random variable that represents the cost of the vertex that the token reaches after k steps, starting from vertex $v_m$.
Then, Lemmas 1 and 2 still hold.
\begin{theorem}
The explicit formula for $u$ is
$$u(t, v_m)=\sum_{s=t}^T \E [C_{T-s, m}].$$
\end{theorem}
\begin{proof}
We proceed by induction on time left.
We observe that 
$$u(T, v_m)=c(v_m) = \E[C_{0,m} ]. $$ 
and
$$u(T-1, v_m)= c(v_m) + \frac{1}{|N_{v_m}|}\sum_{v_j\in N_{v_m}}u(T, v_j)= \E[C_{0,m}+C_{1,m}]. $$
Suppose
$$u(S, v_i)= \sum_{s=S}^T \E[C_{s,i}] $$
for all $i=0, \dots, N-1$ and $S=t,\dots, T.$ 
Observe that for any vertex $v_m,$ because all successors of $v_m$ are equally likely to be visited, we have
\begin{equation}\label{fact}
 \frac{1}{|N_{v_m}|}\sum_{v_j\in N_{v_m}} \E[C_{t,j}] = \E[C_{t-1,m}].
\end{equation}
Then, using \eqref{fact}, we obtain

\begin{align}
u(t-1, v_m) =& c(v_m) +  \frac{1}{|N_{v_m}|}\sum_{v_j\in N_{v_m}}u(t, v_j) = \\ \nonumber 
=& c(v_m) +  \frac{1}{|N_{v_m}|}\sum_{v_j\in N_{v_m}}\sum_{s=t}^T \E[C_{T-s,j}] = \\ \nonumber
=& \E[C_{0,m}] + \sum_{s=t}^T \frac{1}{|N_{v_m}|}\sum_{v_j\in N_{v_m}}\E[C_{T-s,j}]\\ \nonumber
=&  \E[C_{0}] + \sum_{s=t}^T\E[C_{T-s-1,m}]. \nonumber
\end{align}
Thus, the statement is true for $t-1$, too. By induction, the statement holds for any time $t\leq T$. 

\end{proof}
Let us denote with $z_{j,m,T-s}$ the number of ways the token can move from vertex $v_m$ to vertex $v_j$ in precisely $T-s$ steps.
We have the following corollary for $k-$regular graphs:
\begin{corollary}
For a $k-$regular graph, the the value function can be expressed through the following succinct formula
$$u(t, v_m)=\sum_{s=t}^T\frac{1}{k^{T-s}}\Big(\sum_{j=0}^{N-1}c(v_j) z_{j, m, T-s} \Big).$$
\end{corollary}
\begin{proof}
Observe that $|N_{v_m}=k|$ and therefore $\E[C_{s,m}|v_m]=\frac{1}{N^{T-s}}\sum_{j=0}^{N-1}c(v_j) z_{j,m,T-s}.$
\end{proof}

\section{Conclusion}
This paper considers repeated two-person games on a deBruijn graph on any number of symbols, and any length of words. We provide closed form formulas for the optimal strategies associated to these games. The games help analyze an overdetermined linear system of equations, and aid in the proof of the existence of a unique solution to this overdetermined system. 
The existence of an assignment so that all cycles have the same average weight is unexpected, because there are many more cycles than variables as weights on edges. The degrees of freedom that we have are the number of all vertices, times $(n-1)$ (because of the conditions that the sum of the weights of the $n$ edges coming out of the same vertex is zero). This is precisely $(n-1)n^d.$

Thus, for any cycle, we have a linear combination of $c(m), f(m,m|\ell), c(m|\ell)...$ This linear combination has coefficients $1$ or $0,$ based upon whether a vertex (or an edge) is a part of the cycle or not. This produces an overdetermined linear system of equations, so a priori it is not clear that such a system has a solution.

An open question is how the games' results can be extended to general directed graphs. 
Another question is what kind of related adversarial games exist on deBruijn graphs, that have closed form solutions.

\bibliography{ref}

\bibliographystyle{abbrv}

\end{document}